\input xy \xyoption{all}
\documentclass[11pt]{amsart}
\usepackage{amscd,amssymb,amsmath,latexsym,graphicx,amsxtra}
\newtheorem{theorem}{Theorem}[section]

\def\S{{\mathbb S}}
\def\Z{{\mathbb Z}}
\theoremstyle{definition}

\newtheorem{question}[theorem]{Question}

\newtheorem{proposition}[theorem]{Proposition}
\newtheorem{corollary}[theorem]{Corollary}
\newtheorem{remark}[theorem]{Remark}
\numberwithin{equation}{section}
\begin{document}
\title{Exponents of $[\Omega(\S^{r+1}), \Omega (Y)]$}
\author{Marek Golasi\'nski}
\address{Faculty of Mathematics
and Computer Science, University of Warmia and Mazury,
S\l oneczna 54 Street, 10-710 Olsztyn, Poland}
\email{maregk@matman.uwm.edu.pl}
\author{Daciberg Lima Gon\c calves}
\address{Dept. de Matem\'atica - IME - USP,
Rua~do~Mat\~ao 1010 CEP: 05508-090, S\~ao Paulo - SP, Brasil}
\email{dlgoncal@ime.usp.br}
\author{Peter Wong}
\address{Department of Mathematics, Bates College, Lewiston, ME 04240, U.S.A.}
\email{pwong@bates.edu}

\thanks{}

\begin{abstract}
We investigate the exponents of the total Cohen groups $[\Omega(\S^{r+1}), \Omega(Y)]$ for any $r\ge 1$. In particular, we show that for $p\ge 3$, the $p$-primary exponents of $[\Omega(\S^{r+1}), \Omega(\S^{2n+1})]$ and  $[\Omega(\S^{r+1}), \Omega(\S^{2n})]$ coincide with the $p$-primary homotopy exponents of spheres $\S^{2n+1}$ and $\S^{2n}$, respectively.
\par We further study the exponent problem
when $Y$ is a space with the homotopy type of $\Sigma(n)/G$ for a homotopy $n$-sphere $\Sigma(n)$, the complex projective space $\mathbb{C}P^n$ for $n\ge 1$ or the quaternionic projective space $\mathbb{H}P^n$ for $1\le n\le \infty$.
\end{abstract}
\date{\today}
\maketitle

\section{Background and Preliminaries}
Let $X$ be a pointed connected topological space. For any prime $p$, write $\pi_k(X;p)$ for the $p$-primary component of the $k$-th homotopy
group $\pi_k(X)$ for $k\ge 1$. The {\it $p$-primary homotopy exponent} of $X$ is the least integer $\exp_p(X)=p^t$, if it exists, so that $\alpha^{p^t} =1$
for all  elements $\alpha$ in $\pi_{\ast}(X;p)$. If such an integer does not exist or there are no $p$-torsions in $\pi_{\ast}(X)$, we set
$\exp_p(X)=1$. The homotopy exponent problem for spheres has been studied extensively and has a long history.
First, James \cite{james1} has shown that $\exp_2(\S^{2n+1})\le 2^{2n}$ then Toda has proved that $\exp_p(\S^{2n+1})\le p^{2n}$ for $p>2$
and Selick \cite{selick1} has shown that  $\exp_p(\S^3)=p$.

For $p=2$, an upper bound  $\exp_2(\S^{2n+1})\le 2\exp_2(\S^{2n-1})$ has been obtained by Selick \cite{selick}
which combined with James' result \cite{james1} yields $\exp_2(\S^{2n+1})\le 2^{(\frac{3}{2}n)+\varepsilon}$,
where $\varepsilon=\begin{cases}0\;\mbox{if}\;n\equiv 0\,(\bmod\,2),\\
\frac{1}{2}\;\mbox{if}\;n\equiv 1\,(\bmod\,2).
\end{cases}$
\par We notice that the result of Gray \cite{gray} gives a lower bound on the best exponent which combines with the above shows that
$\exp_p(\S^{2n+1})=p^n$ is the best exponent for $p>2$.
\par After that Cohen-Moore-Neisendorfer showed in their seminal work \cite{CMN} that $\exp_p(\S^{2n+1})=p^n$ for $p>3$ and then subsequently,
Neisendorfer \cite{N1} established the same exponent for $p=3$.
\par  Given a group $G$, the {\it $p$-primary exponent} of $G$ is the least integer $\exp_p(G)=p^t$, if it exists, such that ${\alpha}^{p^t}=1$ for any $p$-torsion element
$\alpha$ in $G$. Again, if such an integer does not exist or there are no $p$-torsions in $G$, we set
$\exp_p(G)=1$.

\par The so-called {\em total Cohen groups} $[J(X), \Omega(Y)]$ play an important role in various aspects of the classical non-stable homotopy theory. First introduced by F. Cohen as a tool to tackle the Barratt conjecture, these Cohen groups have found many connections to other aspects in topology and algebra. For instances,
the Cohen groups were used in \cite{SW} to show that the functorial homotopy decompositions of loop space of co-H-spaces is equivalent to the functorial coalgebra decompositions of tensor algebra functor.  This establishes a fundamental connection between homotopy theory (of loops on co-H-spaces) and the modular representation theory of Lie powers. In \cite{Wu1}, modified Cohen groups were used to give bounds on homotopy exponents of $\Omega^2(\S^n_{(2)})$. Moreover, the Cohen groups were used in \cite{Wu2} to produce a ring  so that double loop spaces are modules over this ring in the homotopy category. 

The main objective of this paper is to study  $p$-primary exponents of the total Cohen groups $[J(\S^r), \Omega(Y)]=[\Omega (\S^{r+1}),
\Omega(Y)]$. Here, $J(X)=\mbox{colim}_{n\ge 1}J_n(X)$ denotes the James construction of $X$. It is known that $[J(\S^1), \Omega(Y)]$ is in a one-to-one correspondence with the direct product
$\displaystyle{\prod_{i\ge 2}\pi_i(Y)}$ {\it as sets} but the group structure of $[J(\S^1),\Omega(Y)]$ is far from being abelian.
In fact, $\pi_i(Y)$ is {\it not} even a subgroup of $[J(\S^1),\Omega(Y)]$ in general. Nevertheless, it is natural to ask how the
$p$-torsion elements of $\pi_{\ast}(Y)$ are related to the $p$-torsion elements of $[J(\S^1),\Omega(Y)]$.
In particular, we study $\exp_p([J(\S^r), \Omega(Y)])$ when $Y$ is the $n$-th sphere $\S^n$, a space with the homotopy type of $\Sigma(n)/G$
for a homotopy $n$-sphere $\Sigma(n)$ with a free action of a discrete group $G$, the complex projective space $\mathbb{C}P^n$ for $n\ge 1$ or
the quaternionic projective space $\mathbb{H}P^n$  for $1\le n\le \infty$.

Throughout the rest of this paper we do not distinguish between a map and its homotopy class and we follow freely  notations from the book of Toda \cite{Toda}.

\par This paper is organized as follows. Section 1 recalls some known results on the homotopy exponents of spheres, iterated Whitehead
products of spheres, and the group structure of $[\Omega (\S^{r+1}), \Omega(Y)]$ for $r\ge 1$.
\par Section 2 is devoted to proving the main result on the exponents of $[\Omega (\S^{r+1}), \Omega(\S^N)]$. More precisely, we prove:

\noindent
{\bf Theorem \ref{main}.}  {\em Let $p$ be a prime.
\begin{enumerate}
\item If $p\ne 2$ then $$\exp_p([\Omega (\S^{r+1}), \Omega (\S^N)])=\exp_p(\S^N).$$
\item If $p=2$ then
$$\exp_2([\Omega (\S^{r+1}), \Omega (\S^N)]) \left\{
\aligned
& =\exp_2(\S^N) & \text{if $N$ is odd;} \\
& \le 2\exp_2(\S^N) & \text{if $N$ is even.}
\endaligned
\right.
$$
\end{enumerate}}

In Section 3, we further investigate the exponents of $[\Omega (\S^{r+1}), \Omega(Y)]$ when $Y$ is a space with the homotopy type of $\Sigma(n)/G$,
the complex projective space $\mathbb{C}P^n$ or the quaternionic projective space $\mathbb{H}P^n$ (including $\mathbb{H}P^{\infty}$).
These results are stated in Theorems \ref{form}, \ref{complex}, \ref{quater} and \ref{quaternionic}.

\subsection{Homotopy exponents}
After the work of James \cite{james}, the first major breakthrough in the homotopy exponent problem is the result of
Cohen-Moore-Neisendorfer \cite{CMN}. Shortly thereafter, Neisendorfer \cite{N1} obtained the same result for prime $p=3$.
For odd spheres, we have the following:

\begin{theorem}[\cite{CMN}, \cite{N1}] If $p$ is an odd prime then
$$
\exp_p(\S^{2n+1})=p^n
$$
for any $n\ge 1$.
\end{theorem}

For $p=2$, James \cite{james} showed that the odd spheres $\mathbb{S}^{2n+1}$ have the $2$-primary exponent 
less than equal to $4^n$. Moreover, he showed that the $2$-primary homotopy exponent increases at most by a factor of $4$ as one passes from $\mathbb{S}^{2n-1}$ to $\mathbb {S}^{2n+1}$. Selick \cite{selick} showed
that the $2$-primary homotopy exponent increases at most by a factor of $2$ as one passes from $\mathbb{S}^{4n-1}$ to $\mathbb {S}^{4n+1}$, thereby improving on the upper bound for $\exp_2(\S^{2n+1})$ previously obtained by James \cite{james} and by Cohen \cite{C}.

James has shown the existence of a weak homotopy equivalence $$\Sigma J(X)\simeq \bigvee_{i=1}^\infty\Sigma X^{(i)},$$
where $X^{(i)}$ denotes the $i$-fold smash power of $X$. For the $n$-th sphere $X=\S^n$ one has $\Omega(\S^{n+1})\simeq J(\S^n)$ and the splitting map above leads to the projection maps $\Sigma\Omega\S^{n+1}\to \S^{kn+1}$ for $k\ge 0$ which are adjoint to the maps $$H_k : \Omega(\S^{n+1})\to \Omega(\S^{kn+1})$$
known as the {\em James-Hopf} maps.
\par Write $X_{(p)}$ for the localization of a topological space $X$ at the prime $p$. Then, recall the fibration
$$\mathbb{S}^n_{(2)}\stackrel{E}{\rightarrow} \Omega (\mathbb{S}^{n+1}_{(2)})\stackrel{H_2}{\rightarrow} \Omega (\mathbb{S}^{2n+1}_{(2)})$$
found by James \cite{james} and the fibrations
$$\widehat{\mathbb{S}}^{2n}_{(p)}\rightarrow \Omega (\mathbb{S}^{2n+1}_{(p)})\rightarrow \Omega (\mathbb{S}^{2np+1}_{(p)}),$$
and $$\mathbb{S}^{2n-1}_{(p)}\rightarrow \Omega (\widehat{\mathbb{S}}^{2n}_{(p)})\rightarrow \Omega (\mathbb{S}^{2np-1}_{(p)})$$
found by Toda \cite{Toda} for $p>2$, where $\widehat{\mathbb{S}}^{2n}$ is the $(2np-1)$-skeleton of the loop space $\Omega(\mathbb{S}^{2n+1})$.
Thus, the fibrations above and the Serre result \cite{serre} lead to:
\begin{theorem} \mbox{\em (1)} The fibre of the James-Hopf map $H_2 : \Omega(\S^{2n})\to\Omega(\S^{4n-1})$ is $\S^{2n-1}$ and there is an odd primary equivalence $($due to Serre$)$
$$\Omega(\S^{2n})\simeq\S^{2n-1}\times\Omega(\S^{4n-1}).$$\label{raven}
\mbox{\em (2)} The $p$-local fibre of $H_p : \Omega(\S^{2n+1})\to \Omega(\S^{2pn+1})$ is $J_{p-1}(\S^{2n})$ for any prime $p$ $($due to James for $p=2$ and
Toda for $p>2$$)$.
\end{theorem}

For even-dimensional spheres, Theorem \ref{raven}(1) gives the torsions at odd primes $p$ in terms of those of odd-dimensional spheres, using
\begin{equation}\label{expon}\pi_m(\mathbb{S}^{2n};p)\cong \pi_{m-1}(\mathbb{S}^{2n-1};p)\oplus \pi_m(\mathbb{S}^{4n-1};p)
\end{equation}
for $m\ge 2n$, where $\pi_m(X;p)$ stands for the $p$-primary homotopty component of the $m$-th homotopy group $\pi_m(X)$ of a space $X$.
This implies that $\mathbb{S}^{2n}$ has the $p$-primary homotopy exponent $p^{2n-1}$.

The $EHP$ sequences associated to the fibration
$$\mathbb{S}^{2n-1}_{(2)}\stackrel{E}{\rightarrow} \Omega (\mathbb{S}^{2n}_{(2)})\stackrel{H}{\rightarrow} \Omega (\mathbb{S}^{4n-1}_{(2)})$$
and James' result \cite{james} for $p=2$ show that the $2$-primary homotopy exponent $\exp_2(\S^{2n})$ is bounded by $4^{2n}$.
\par For spaces other than the spheres, Neisendorfer \cite[Corollary 0.2]{N2} showed that $\exp_p(P^n(p^r))\le p^{r+1}$ 
for Moore spaces $P^n(p^r)$ of type $(\mathbb Z_{p^r},n-1)$ with an odd prime $p$ and $n\ge 3$. For prime $p=2$, the problem has been investigated by Theriault \cite{T}.
Relying on results by James \cite{james1} and Toda \cite{Toda1}, Stanley reproved in \cite{stanley} Long's result \cite{long} that finite $H$-spaces have an exponent at any prime.
More recently, the homotopy exponent problem has also been studied for certain homogeneous spaces (see e.g., \cite{GZ}, \cite{ZZS}).

\subsection{Iterated Whitehead products of spheres}
Given $\alpha\in\pi_k(X)$ and $\beta\in\pi_l(X)$ with $k,l\ge 1$, write $[\alpha,\beta]\in \pi_{k+l-1}(X)$ for their Whitehead product.
\begin{proposition} {\em Let $\iota_n$ be the identity map of the $n$-sphere $\S^n$. Then:\label{wh}
\begin{enumerate}
\item $[\iota_n, \iota_n]$ has infinite order if $n$ is even, is trivial if $n =1,3,7$, and has order $2$ otherwise \mbox{\em (\cite[ (1.2)]{GM1})};
\item $[[\iota_n, \iota_n],\iota_n]$ has order  $3$ if $n$ is even  and is trivial otherwise \mbox{\em(\cite[Lemma 1.2 and (1.4)]{GM1})};
\item all  Whitehead products in $\iota_n$ of weight  $\geq 4$ vanish \mbox{\em (\cite[Chapter XI, (8.8) Theorem]{Wh})};
\item if $\alpha\in\pi_k(X)$, $\beta\in\pi_l(X)$ and $[\alpha,\beta]=0$ then $[\alpha\circ\delta,\beta\circ\delta^\prime]=0$ for $\delta\in\pi_s(\S^k)$ and $\delta^\prime\in\pi_t(\S^l)$ \mbox{\em (\cite[Chapter X,  (8.14) Theorem]
{Wh})};

\item  $($Jacobi identity$)$ If  $\alpha\in \pi_{p+1}(X)$, $\beta \in \pi_{q+1}(X)$, $\gamma\in \pi_{r+1}(X)$,  and $p$, $q$, $r$ are all
positive, then
$$(-1)^{(p+1)(r+1)}[[\alpha, \beta], \gamma] + (-1)^{(q+1)(p+1)}[[\beta, \gamma], \alpha] + (-1)^{(q+1)(r+1)}[[\gamma, \alpha], \beta]=0$$  \mbox{\em (\cite[Chapter X,  (7.14) Corollary]{Wh})};

\item if $\alpha\in\pi_k(X)$, $\beta\in\pi_l(X)$ and $\delta\in\pi_s(\S^{k-1})$, $\delta'\in\pi_t(\S^{l-1})$
then $[\alpha\circ \Sigma\delta,\beta\circ \Sigma \delta']=[\alpha,\beta]\circ \Sigma(\delta\wedge\delta')$ \mbox{\em (\cite[Chapter X, (8.18) Theorem]{Wh}).}
\end{enumerate}}
\end{proposition}

To state the next results on Whitehead products, first recall from \cite[Chapter II]{baues} that the \textit{exterior cup products} are pairings
$$\sharp, \,\underline{\sharp}\colon[\Sigma X, \Sigma A]\times[\Sigma Y, \Sigma B]\longrightarrow [\Sigma X\wedge Y, \Sigma A\wedge B]$$
defined by the compositions $$\alpha\sharp\beta\colon \Sigma X\wedge Y\stackrel{\alpha\wedge Y}{\to}\Sigma A\wedge Y=A\wedge \Sigma Y
\stackrel{A\wedge\beta}{\to} A\wedge \Sigma B$$
and
$$\alpha\underline{\sharp}\beta\colon \Sigma X\wedge Y=X\wedge \Sigma Y\stackrel{X\wedge\beta}{\to}X\wedge \Sigma B=\Sigma X\wedge B
\stackrel{\alpha\wedge B}{\to} \Sigma A\wedge B,$$
respectively for $(\alpha,\beta)\in [\Sigma X, \Sigma A]\times[\Sigma Y, \Sigma B]$, where $\alpha\wedge Y$ is the map $\alpha\wedge\text{id}_Y$ and $A\wedge\beta$ is the
map $\text{id}_{\Sigma A}\wedge\beta$, up to
the shuffle of the suspension coordinate. These products are associative.
\par Let  $h_k\colon \pi_m(\S^{n+1}) \to   \pi_m(\S^{kn+1})$ be the map induced by the James-Hopf $k$-invariant $H_k : \Omega(\S^{n+1})\to \Omega(\S^{kn+1})$ for $k>1$.
Then, by \cite[Chapter III,  (1.4) Proposition and (1.5) Corollary]{baues}, we have:
\begin{proposition}\label{baues}   {\em   Let $\alpha_i\in \pi_{m_i}(\S^n)$  for  $i=1,2,3$. Then we have the following formulas:
\begin{equation*}
\begin{aligned}
&\mbox{\em (1)}\;\;\;\; [\alpha_1, \alpha_2]=[\iota_n,\iota_n] \circ (\alpha_1 \underline{\sharp} \alpha_2)+(-1)^{n-1}[\iota_n,[\iota_n,\iota_n]]\circ ((h_2\alpha_1)\underline \sharp\alpha_2)+[\iota_n,[\iota_n,\iota_n]]\circ (\alpha_1\underline \sharp h_2\alpha_2), \\
&\mbox{\em (2)}\;\;\;\;[[\alpha_1, \alpha_2], \alpha_3]=[[\iota_n,\iota_n],\iota_n]\circ(\alpha_1\sharp \alpha_2\sharp \alpha_3).
\end{aligned}
\end{equation*}}
\end{proposition}

Next, \cite[ Chapter III, (1.8) and (1.9) Corollaries]{baues} yield the following result proved by Hilton \cite{hilton}.
\begin{proposition} \label{p} {\em Let    $\alpha_1\in \pi_k(\S^m)$ for
$m> 1$ and  $\alpha_2,\alpha_3\in\pi_m(\S^n)$. Then
$$(\alpha_2+\alpha_3)\circ \alpha_1 = \alpha_2\alpha_1 + \alpha_3\alpha_1 + [\alpha_2,\alpha_3]\circ h_2(\alpha_1).$$
\par In particular, if $t\in\Z$ then
$$(t\iota_m)\circ \alpha = t \alpha +\left (\frac{t(t-1)}{2}\right)[\iota_m,\iota_m] h_2(\alpha)$$
for  $\alpha\in \pi_k(\S^m)$.}
\end{proposition}

 Let $[\alpha_1,\ldots, \alpha_k]$ denote any possible iterated Whitehead product of weight $k$
of the elements  $\alpha_1,\ldots, \alpha_k$, where $\alpha_i\in \pi_{m_{i}}(\S^n)$ for $i=1,\ldots,k$.
Applying Propositions \ref{baues} and  \ref{p} we obtain the well-known result:

\begin{proposition}\label{three} {\em Let  $n\geq 2$.

\mbox{\em (1)}  If $\alpha_i\in \pi_{m_{i}}(\S^n)$ with $i=1,2,3$ then
 $3[[\alpha_1, \alpha_2], \alpha_3]=0;$

\mbox{\em (2)} all iterated Whitehead products  $[\alpha_1,\ldots, \alpha_k]$  of weight $k\geq 4$ vanish.}
 \end{proposition}
\begin{proof} (1):  If $n$ is even then by Proposition \ref{wh}(2) we have  $3[[\iota_n, \iota_n], \iota_n]=0$. Hence, for  $\alpha_i\in \pi_{m_i}(\S^n)$
 with $i=1,2,3$, we have
$$0=(3[[\iota_n,\iota_n],\iota_n])\circ(\alpha_1\sharp \alpha_2\sharp\alpha_3)=(2[[\iota_n,\iota_n],\iota_n]+[[\iota_n,\iota_n],\iota_n])\circ(\alpha_1\sharp \alpha_2\sharp\alpha_3).$$
By Proposition \ref{wh}(3) and Proposition \ref{p}, we obtain
\begin{equation*}
\begin{aligned}
(2[[\iota_n,\iota_n],\iota_n]+[[\iota_n,\iota_n],\iota_n])\circ(\alpha_1\sharp \alpha_2\sharp\alpha_3)&=
(2[[\iota_n,\iota_n],\iota_n])\circ(\alpha_1\sharp \alpha_2\sharp\alpha_3)+[[\iota_n,\iota_n],\iota_n]\circ(\alpha_1\sharp \alpha_2\sharp\alpha_3) \\
&=3 ([[\iota_n,\iota_n],\iota_n]\circ(\alpha_1\sharp \alpha_2\sharp\alpha_3)).
\end{aligned}
\end{equation*}
Since  $[[\alpha_1, \alpha_2], \alpha_3]=[[\iota_n,\iota_n],\iota_n]\circ(\alpha_1\sharp \alpha_2\sharp\alpha_3)$
by  Proposition \ref{baues}(2),  assertion (1) follows.

(2): Let $k=4$. Then, by Proposition \ref{wh}(3), we have $[[[\iota_n,\iota_n],\iota_n],\iota_n]]=0$.
Using Proposition \ref{wh}(4), we deduce that $[([[\iota_n,\iota_n],\iota_n])\circ(\alpha_1\sharp\alpha_2\sharp\alpha_3),\alpha_4]]=0$.
Hence, Proposition \ref{baues}(2) implies $[[[\alpha_1,\alpha_2],\alpha_3],\alpha_4]]=0$.

Furthermore, by the Jacobi identity of Proposition\ref{wh}(5), we have $$[[\alpha_1,\alpha_2],[\alpha_3,\alpha_4]]=
\pm [[[\alpha_1,\alpha_2],\alpha_3],\alpha_4]]\pm  [[[\alpha_1,\alpha_2],\alpha_4],\alpha_3]].$$
Thus, we conclude that $[[\alpha_1,\alpha_2],[\alpha_3,\alpha_4]]=0$.
Now the rest of the proof is an inductive argument.

Suppose that the result is true for Whitehead products of weight $m$ and let us show for $m+1$, where we will assume that $m+1\geq 5$. Since
$[\alpha_1,\ldots,\alpha_{m+1}]$ is of the form $[\theta_1, \theta_2]$, where $\theta_1$, $\theta_2$  are iterated Whitehead products of weight $k_1$ and  $k_2$, respectively with $k_1+k_2=m+1$.
If  one of the $k_{i}$'s is $\geq 4$, then the result follows by inductive hypothesis.  Otherwise   $k_1, k_2<4$, and we have two cases,
namely both are equal to $3$ or one is $3$ and the other is $2$. In both cases the result follows by inductive hypothesis and the Jacobi identity from Proposition \ref{wh}(5).
\end{proof}

The next result is probably also well-known to the experts, but nevertheless, we have decided to state:
\begin{proposition}\label{2times} {\em Let $n$ be odd and  $\alpha_1\in\pi_k(\S^n)$, $\alpha_2\in\pi_l(\S^n)$ and
$\alpha_3\in\pi_t(\S^n)$. Then:

\mbox{\em (1)} $2[\alpha_1,\alpha_2]=0$;

\mbox{\em (2)} $[[\alpha_1, \alpha_2], \alpha_3]=0.$}
\end{proposition}
\begin{proof} (1): Let $\alpha_i\in\pi_{m_i}(\mathbb{S}^n)$ for $i=1,2$.
By Proposition \ref{wh}(2), we have $[[\iota_n,\iota_n],\iota_n]=0$.
Hence, Proposition \ref{baues}(1) implies $[\alpha_1,\alpha_2]=[\iota_n,\iota_n]
\circ(\alpha_1\underline{\sharp}\alpha_2)$. On the other hand, Proposition \ref{wh}(1) gives $2[\iota_n,\iota_n]=0$. Then, we have
$(2[\iota_n,\iota_n])\circ(\alpha_1\underline{\sharp}\alpha_2)=0$.

It follows from Proposition \ref{baues}(1) that
$$0=(2[\iota_n,\iota_n])\circ(\alpha_1\underline{\sharp}\alpha_2)=2(([\iota_n,\iota_n])\circ(\alpha_1\underline{\sharp}\alpha_2))+[[\iota_n,\iota_n],[\iota_n,\iota_n]]\circ
h_2(\alpha_1\underline{\sharp}\alpha_2).$$
But, the Jacobi identity of Proposition \ref{wh}(5) together with the fact that $[[\iota_n,\iota_n],\iota_n]=0$
yield $[[\iota_n,\iota_n],[\iota_n,\iota_n]]=0$.
Consequently, $2[\alpha_1,\alpha_2]=0$.

We also present another proof of assertion (1) as follows.
In view of Proposition \ref{wh}(1) we have $2[\iota_n,\iota_n]=[2\iota_n,\iota_n]=0$.
Then, Proposition \ref{wh}(4) leads to $[2\iota_n,\alpha_2]=[\iota_n,2\alpha_2]=0$.
Again, by Proposition \ref{wh}(4), we get $[\alpha_1,2\alpha_2]=2[\alpha_1,\alpha_2]=0$
and the assertion follows.

(2): If $\alpha_i\in\pi_{m_i}(\mathbb{S}^n)$ for $i=1,2,3$ then, it follows from
Proposition \ref{three}(1) that $3[[\alpha_1,\alpha_2],\alpha_3]=0$. Together with (1), we conclude that
$[[\alpha_1,\alpha_2],\alpha_3]=0$ and the proof is complete.
\end{proof}

\subsection{Group structure on $[\Omega(\S^{r+1}), \Omega (Y)]$}
By the group isomorphism $[\Sigma J(\S^r),Y] \cong [\Omega(\S^{r+1}), \Omega (Y)]$, it follows from $\displaystyle{\Sigma J(\S^r) \simeq \bigvee_{i\ge 1} \S^{ir+1}}$
that $[\Omega(\S^{r+1}), \Omega (Y)]$ is in a one-to-one correspondence with the product $\displaystyle{\prod_{i\ge 1}\pi_{ir+1}(Y)}$ {\it as sets}.
Following \cite{ggw5}, we recall how the group multiplication $\circledast$ on $[\Sigma J(\S^r),Y]$ is defined.

Identifying $[\Omega(\S^{r+1}), \Omega (Y)]$ with $\displaystyle{\prod_{i\ge 1}\pi_{ir+1}(Y)}$ as sets, a typical element of $[\Omega(\S^{r+1}), \Omega (Y)]$ is an infinite tuple
${\bar{\alpha}}=(\alpha_1, \alpha_2,\dots)$, where $\alpha_i\in \pi_{ir+1}(Y)$ for $i\ge 1$. Denote by $(\bar{\alpha})_j$ the $j$-th coordinate of
$\bar{\alpha}\in [\Omega(\S^{r+1}), \Omega (Y)]$, i.e., $(\bar{\alpha})_j=\alpha_j \in \pi_{jr+1}(Y)$. Let $\bar{\alpha}=(\alpha_1, \alpha_2,\ldots), \, \bar{\beta}=(\beta_1, \beta_2,\ldots)$ be two elements in
$[\Omega(\S^{r+1}), \Omega (Y)]$. Then, the product $\bar{\alpha} \circledast \bar{\beta}$ is defined to be the element whose $j$-th coordinate is given by
\begin{equation}\label{multiplication}
(\bar{\alpha} \circledast \bar{\beta})_j=\alpha_j+\beta_j+\sum_{i+s =j}\Phi_{i,i+s-1} [\alpha_i, \beta_s],
\end{equation}
where the coefficient $\Phi$ is defined as follows. For any positive integers $l,k$ with $1\le l\le k$,
$$
\Phi_{l,k} \quad = \quad
\left\{
\aligned
& -\binom{\frac{k}{2}}{\frac{l}{2}} \qquad & \text{if $l$ is even and $k$ is even;} \\
& 0 & \text{if $l$ is odd and $k$ is even;} \\
& \binom{\frac{k-1}{2}}{\frac{l-1}{2}} & \text{if $l$ is odd and $k$ is odd;} \\
& -\binom{\frac{k-1}{2}}{\frac{l}{2}} & \text{if $l$ is even and $k$ is odd.}
\endaligned
\right.
$$

Using the multiplication \eqref{multiplication}, the coordinates of any $p$-torsion element of $[\Omega(\S^{r+1}), \Omega (Y)]$ are $p$-torsion elements of $\pi_{\ast}(Y)$ as we show in:

\begin{proposition} {\em \label{tor}
Let $\bar{\alpha}=(\alpha_1, \alpha_2,\ldots) \in [\Omega(\S^{r+1}), \Omega (Y)]$. If $\bar{\alpha}$ is a $p$-torsion element then each $\alpha_i$ is a $p$-torsion element of $\pi_{\ast}(Y)$.}
\end{proposition}
\begin{proof}
Suppose $\bar{\alpha}^k=1$ for some $k=p^t$. It follows from the multiplication $\circledast$ that $(\bar{\alpha}^k)_1=k\alpha_1$.
This implies that $\alpha_1$ is a $p$-torsion element of $\pi_{r+1}(Y)$. We proceed by induction as follows. Suppose $\alpha_i$
is an element in $\pi_{\ast}(Y;p)$ for $i<n$. Again, using $\circledast$, $(\bar{\alpha}^{k})_n$ is the sum of $k\alpha_n$ and integer
multiples of elements of the form $[\alpha_i,\gamma]$, where $\gamma \in \pi_{jr+1}(Y)$ with $i+j=n$. By inductive hypothesis, $\alpha_i$
is an element in $\pi_{\ast}(Y;p)$ and so is $[\alpha_i, \gamma]$. Since $\bar{\alpha}^k=1$, we conclude that $\alpha_n$ must also be an element in $\pi_{nr+1}(Y;p)$
and the proof follows.
\end{proof}

Let $\bar{\alpha}, \bar{\beta}\in [\Omega(\S^{r+1}), \Omega(\S^N);p]$ and write $\bar{\gamma}=\bar{\alpha}\circledast \bar{\beta}\in [\Omega(\S^{r+1}), \Omega(\S^N)]$.
 By \eqref{multiplication}, we have $\gamma_j=\alpha_j+\beta_j+\sum_{i+s =j}\Phi_{i,i+s-1} [\alpha_i, \beta_s]$.
 Suppose $\bar{\alpha}^{p^{t_1}}=1=\bar{\beta}^{p^{t_2}}$. Then, one can easily show that  $\gamma^{p^{\max\{t_1,t_2\}}}=1$.
 Furthermore, $\bar{\alpha}^{p^{t_1}}=1$ implies $(\bar{\alpha}^{-1})^{p^{t_1}}=1$.

Given a prime $p$, spaces $X,Y$ and a co-$H$-structure on $X$, write $[X,Y;p]$ for the set of $p$-primary components
of $[X,Y]$. Then, using Proposition \ref{tor}, it is straightforward to show the following
\begin{proposition} {\em  \mbox{\em (1)} The bijection $[\Omega(\S^{r+1}),\Omega(\S^{N})]\stackrel{}{\longleftrightarrow}\prod_{i\ge 1}\pi_{n_i}(\S^{N})$
restricts to a bijection
$$[\Omega (\S^{r+1}),\Omega(\S^{N});p]\stackrel{}{\longleftrightarrow}\prod_{i\ge 1}\pi_{n_i}(\S^{N};p) \qquad \text{for $p\not=2$};$$

$\mbox{\em (2)}\;\;\;\;[\Omega(\S^{r+1}),\Omega(\S^N);p]$
is a subgroup of $[\Omega(\S^{r+1}),\Omega(\S^N)]$ for any prime $p\ge 2$.}
\end{proposition}

\section{Exponents of $[\Omega (\S^{r+1}), \Omega (\S^N)]$}
Before we state the main theorem of this section, first examine the coordinates of the powers $\bar{\gamma}^M$ for $\bar{\gamma}\in [\Omega(\S^{r+1}),\Omega (\S^N)]$.
Suppose $\alpha \in \pi_{n+1}(\S^N), \beta \in \pi_{m+1}(\S^N)$. Consider the element $(\alpha, \beta)\in [\Omega(\S^{r+1}),\Omega (\S^N)]$ as an infinite sequence with only two non-zero coordinates in positions $n$ and $m$.
Then,
$$
(\alpha, \beta)^2=(2\alpha, 2\beta, \Phi_{n,2n-1}[\alpha,\alpha], \Phi_{m,2m-1}[\beta, \beta], \Phi_{\Delta} [\alpha, \beta]).
$$
Here $\Phi_{\Delta}=\Phi_{n,n+m-1}+(-1)^{(n+1)(m+1)}\Phi_{m,n+m-1}$.

Now, we compute $(\alpha, \beta)^3$. The possible non-zero coordinates are divided into three types: (homogenous) elements, Whitehead products
and triple Whitehead products. This is the case because for spheres, all quadruple Whitehead products vanish by Proposition \ref{three}(2).

It is easy to see that the first type consists of $3\alpha$ and $3\beta$. Second type coordinates are: $3\Phi_{n,2n-1}[\alpha, \alpha]$, $3\Phi_{m,2m-1}[\beta,\beta]$, and $3\Phi_{\Delta}[\alpha,\beta]$. Finally, for the triple Whitehead products, we have:

$$
\Phi_{n,2n-1}\Phi_{2n,3n-1}[[\alpha, \alpha],\alpha], \Phi_{n,2n-1}\Phi_{2n,2n+m-1}[[\alpha,\alpha],\beta],
$$
$$
\Phi_{m,2m-1}\Phi_{2m,2m+n-1}[[\beta,\beta],\alpha], \Phi_{m,2m-1}\Phi_{2m,2m+n-1}[[\beta,\beta],\beta],
$$
$$
\Phi_{\Delta}\Phi_{n+m,2n+m-1}[[\alpha,\beta],\alpha], \Phi_{\Delta}\Phi_{n+m,2m+n-1}[[\alpha,\beta],\beta].
$$

To continue in this fashion, the power $(\alpha, \beta)^M$ has the following non-zero coordinates:

Type I: $M\alpha, M\beta$;

Type II: $(1+2+3+\ldots+ (M-1))\Phi_{n,2n-1}[\alpha,\alpha]=\binom{M}{2} \Phi_{n,2n-1}[\alpha,\alpha]$, $\binom{M}{2}\Phi_{m,2m-1}[\beta,\beta]$, $\binom{M}{2}\Phi_{\Delta}[\alpha,\beta]$;

Type III: each of the six triple Whitehead products will have the following as a factor in its coefficient:
$$
1+3 + (1+2+3) + (1+2+3+4) + \cdots +(1+2+\cdots+ M-1).
$$
This sum in turn is equal to
$$\sum_{l=2}^{M-1} \binom{l}{2}.$$
By \cite[p.\ 120]{ggw4}, we have
$$\sum_{l=2}^{M-1} \binom{l}{2}=\binom{M}{3}.$$

Therefore, for an arbitrary element $\bar{\alpha}=(\alpha_1, \alpha_2,\ldots) \in [\Omega(\S^{r+1}),\Omega (\S^N)]$, the non-zero coordinates of $\bar{\alpha}^M$ also fall into these three types as described above.
\par Now, the main result of this section is the following:
\begin{theorem}\label{main}
Let $p$ be a prime.
\begin{enumerate}
\item If $p\ne 2$ then $$\exp_p([\Omega (\S^{r+1}), \Omega (\S^N)])=\exp_p(\S^N).$$
\item If $p=2$ then
$$\exp_2([\Omega (\S^{r+1}), \Omega (\S^N)]) \left\{
                                                                             \aligned
                                                                             & =\exp_2(\S^N) & \text{if $N$ is odd;} \\
                                                                             & \le 2\exp_2(\S^N) & \text{if $N$ is even.}
                                                                             \endaligned
                                                                             \right.
                                                                             $$
\end{enumerate}
\end{theorem}
\begin{proof}
Suppose $p^t=\exp_p(\S^N)$ and $\bar{\alpha}=(\alpha_1, \alpha_2,\ldots)$ is a $p$-torsion element in $[\Omega(\S^{r+1}),\Omega (\S^N)]$.

Case (1): $p$ is odd.

Let $M=p^t$. Since $M=\exp_p(\S^N)$, all coordinates of Type I must be zero. But $\binom{M}{2}=\frac{p^t(p^t-1)}{2}$ and $p$ is odd,
so it follows that $p^t \mid \binom{M}{2}$ and consequently $M=p^t$ divides all coefficients of coordinates of Type II in $\bar{\alpha}^M$. Thus, we conclude
that these coordinates must be zero. For Type III coordinates, we consider two subcases.

(i) Case $N$ is odd. By Proposition \ref{2times}, all triple Whitehead products vanish so all Type III coordinates must be zero.

(ii) Case $N$ is even. By Proposition \ref{three}, we have $3[\alpha, [\beta,\gamma]]=0$ for any $\alpha,\beta,\gamma\in \pi_{\ast}(\S^N)$. We conclude that if $p\ne 3$ then
$[\alpha, [\beta,\gamma]]=0$ so there are no non-zero coordinates of Type III in $\bar{\alpha}^M$.

If $p=3$, then $\binom{M}{3}=\frac{3^t(3^t-1)(3^t-2)}{6}$. Since $N$ is even, it follows from \cite{Toda} that $M=3^t$, where $t\ge 2$
provided $N>2$. It follows that $3 \mid \binom{M}{3}$. Again, by Proposition \ref{three}, there are no non-zero coordinates of Type III
in $\bar{\alpha}^M$.

When $N=2$, all triple Whitehead products in $\pi_{\ast}(\S^2)$ vanish so again there are no non-zero coordinates of Type III
in $\bar{\alpha}^M$. Hence, we conclude that for $p\ne 2$, $\exp_p([\Omega (\S^{r+1}), \Omega (\S^N)])=\exp_p(\S^N)$.

Case (2): $p=2$. We consider two subcases:

(i) Case $N$ is odd. For $N>3$, it follows from \cite{selick} that $\exp_2(\S^N)=2^t$ for some $t>1$. Since $N$ is odd, by
Proposition \ref{2times}, it follows that all triple Whitehead products must vanish. In other words, there are no non-zero coordinates
of Type III in $\bar{\alpha}^M$, where $M=2^t$. Since $\binom{M}{2}=2^{t-1} (2^t-1)$ and $t>1$, $2\mid \binom{M}{2}$ so that there are
no non-zero coordinates of Type II as $2[\alpha,\beta]=0$ by Proposition \ref{2times}. For $N=3$, there are no non-trivial Whitehead products
since $\S^3$ is a group. Clearly there are no non-zero coordinates of Type I in $\alpha^M$. Thus for $N$ odd, we have
$\exp_2([\Omega (\S^{r+1}), \Omega (\S^N)])=\exp_2(\S^N)$.

(ii) Case $N$ is even. Again for $N>2$, it follows from \cite{selick} and from the $EHP$ sequence that $\exp_2(\S^N)=2^t$ for some $t>1$. Now let $M=2^{t+1}$.
Then $\binom{M}{2}=2^t(2^{t+1}-1)$ so $\binom{M}{2}[\alpha,\beta]=0$ if $\alpha,\beta$ are elements in $\pi_{\ast}(\S^N;2)$. Thus, we conclude that there
are no non-zero coordinates of Type II. By Proposition \ref{three}, we have $3[\alpha, [\beta,\gamma]]=0$. Thus, if $\alpha,\beta$ and $\gamma$ are elements of $\pi_{\ast}(\S^N;2)$,
it follows that the triple Whitehead product $[\alpha, [\beta,\gamma]]=0$. We conclude that there are no non-zero coordinates of Type III.
Since $M=2 \exp_2(\S^N)$, there are no non-zero coordinates of Type I in $\alpha^M$.

Finally for $N=2$, the same arguments as above show that there are no non-zero coordinates of Types I and II. For Type III coordinates,
we note that all triple Whitehead products in $\pi_{\ast}(\S^2)$ vanish so that there are no non-zero coordinates of Type III.
Hence, for $N$ even, we have $\exp_2([\Omega (\S^{r+1}), \Omega (\S^N)])\le 2 \exp_2(\S^N)$ and the proof is complete.

\end{proof}

\begin{remark}\label{remark1}
When $p=2$ and $N$ is even, we can further analyze the $2$-primary exponent $\exp_2([\Omega (\S^{r+1}), \Omega (\S^N)])$ as follows.
Suppose $N>2$ is even and let $M=2^t=\exp_2(\S^N)$ with $t>1$. For Type II coordinates, we have $\binom{M}{2}=2^{t-1}(2^t-1)$;
$$\Phi_{n,2n-1}=\left\{
                                                                             \aligned
                                                                             & \binom{n-1}{\frac{n-1}{2}} & \text{if $n$ is odd;} \\
                                                                             & -\binom{n-1}{\frac{n}{2}} & \text{if $n$ is even.}
                                                                             \endaligned
                                                                             \right.
                                                                             $$

Now, recall the Lucas' formula:
$$\binom {m} {n}\equiv \prod \binom {m_i}{n_i}\,(\bmod\, p),$$
where $m=m_{k}p^{k}+m_{k-1}p^{k-1}+\cdots +m_{1}p+m_{0}$
and $n=n_{k}p^{k}+n_{k-1}p^{k-1}+\cdots +n_{1}p+n_{0} $
are the base $p$ expansions of $m$ and $n$ respectively. This uses the convention that $\binom{m}{n} = 0$
if $m < n$ for non-negative integers $m$ and $n$ and a prime $p$.

Then, we get $\binom{n-1}{\frac{n-1}{2}}\equiv 0\,(\bmod\,2)$ and  $\binom{n-1}{\frac{n}{2}}\equiv\begin{cases} 1\,(\bmod\,2)\; \text{if $n=2^k$},\\
0\;\text{otherwise}.
\end{cases}$

It follows that $\binom{M}{2}\Phi_{n,2n-1}[\alpha,\alpha]=0$ except possibly when $n=2^k$. In the case when $n=2^k$, we have
$[\alpha,\alpha]=(-1)^{(nr+1)(nr+1)}[\alpha,\alpha]=-[\alpha,\alpha]$. Thus if $\alpha$ is an element in
$\pi_{\ast}(\S^N;2)$, we have $\binom{M}{2}\Phi_{n,2n-1}[\alpha, \alpha]=0$ because
$2 \mid \binom{M}{2}$ since $t>1$. Similarly, $\binom{M}{2}\Phi_{m,2m-1}[\beta,\beta]=0$. Thus, the only possible non-zero coordinates
of Type II are $\binom{M}{2}\Phi_{\Delta}[\alpha,\beta]$.

Now,
$$
\Phi_{\Delta} \quad = \quad
\left\{
\aligned
& -\binom{\frac{n+m}{2}}{\frac{n}{2}}  & \text{if $m$ is even and $n$ is even;} \\
& 2\binom{\frac{n+m-2}{2}}{\frac{n-1}{2}}  & \text{if $m$ is odd and $n$ is odd;} \\
& -\binom{\frac{n+m-1}{2}}{\frac{n}{2}}  & \text{if $m$ is odd and $n$ is even;} \\
& -\binom{\frac{n+m-1}{2}}{\frac{m}{2}} & \text{if $m$ is even and $n$ is odd.}
\endaligned
\right.
$$
Therefore, if both $m$ and $n$ are odd, $\binom{M}{2}\Phi_{\Delta}[\alpha,\beta]=0$. This means that if $\bar{\alpha}=(\alpha_1,\alpha_2,\ldots)
\in [\Omega(\S^{r+1}),\Omega(\S^N)]$ ($N$ even) such that $\alpha_{2i}=0$ for all $i\ge 1$ then $\alpha^M=1$, where $M=\exp_2(\S^N)$.
\end{remark}
Based upon the proof of Theorem \ref{main} and Remark \ref{remark1}, we pose the following:
\begin{question}
Let $\alpha\in\pi_r(\S^{2N};2)$, $\beta\in\pi_s(\S^{2N};2)$ with $\alpha\not=\beta$, $r$ or $s$ is odd.
Suppose that the orders $|\alpha|=|\beta|=2^t$, the $2$-exponent of the sphere $\S^{2N}$.
\par Is it true that the order $|[\alpha,\beta]|<2^t$ for any $\alpha\in\pi_r(\S^{2N};2)$, $\beta\in\pi_s(\S^{2N};2)$ as
above or there are $\alpha\in\pi_r(\S^{2N};2)$,  $\beta\in\pi_s(\S^{2N};2)$ such that
$|[\alpha,\beta]|=2^t$?
\end{question}

\section{{Exponents of  $[\Omega (\S^{r+1}), \Omega (\Sigma(n)/G)]$, $[\Omega (\S^{r+1}), \Omega (\mathbb{C}P^n)]$} and $[\Omega (\S^{r+1}), \Omega (\mathbb{H}P^n)]$}
In this section, we examine $\exp_p([\Omega(\S^{r+1}), \Omega(Y)])$ when $Y$ with the homotopy type is of $\Sigma(n)/G$ for a homotopy $n$-sphere $\Sigma(n)$ with a free action of a discrete group $G$, a
complex projective space $\mathbb{C}P^n$ for $n\ge 1$ or a quaternionic projective space $\mathbb{H}P^n$ for $1\le n\le \infty$. First, we show certain basic properties about exponents.

Note that $\Omega(Y_1\times Y_2)\simeq \Omega(Y_1)\times \Omega(Y_2)$ and the space $\Sigma\Omega(\S^{r+1})$ is $1$-connected for $r\ge 1$.
Then, the following result is straightforward.

\begin{proposition}\label{cover} {\em \mbox{\em (1)} For any prime $p$, we have
$$\exp_p([\Omega(\S^{r+1}),\Omega(Y_1\times Y_2)])=\max\{\exp_p([\Omega(\S^{r+1}),\Omega(Y_1)]), \exp_p([\Omega(\S^{r+1}),\Omega(Y_2)])\};$$

\mbox{\em (2)} A covering map $\tilde X\to X$ induces an isomorphism $$[\Omega(\S^{r+1}),\Omega(\tilde X)]\cong [\Omega(\S^{r+1}),\Omega(X)]$$
and
$$\exp_p([\Omega(\S^{r+1}),\Omega(\tilde X)]=\exp_p([\Omega(\S^{r+1}),\Omega(X)]$$
for any prime $p$.}
\end{proposition}

Recall that a finite dimensional $CW$-complex $\Sigma(n)$ with the homotopy type of the $n$-th sphere $\S^n$ is called a {\em homotopy $n$-sphere}.
If a discrete group $G$ acts freely and properly discontinuously on $\Sigma(n)$ then the quotient map $\Sigma(n)\to \Sigma(n)/G$ is a covering map.
Let $\mathbb{R}P^n$ be the $n$-th  real projective space for $n\ge 1$ and $\gamma_{n,\mathbb{R}} : \S^n\to \mathbb{R}P^n$ be the quotient map.
Then, Proposition \ref{cover}(2) yields:
\begin{theorem} The quotient map $\Sigma(n)\to \Sigma(n)/G$ induces an isomorphism $$[\Omega(\S^{r+1}),\Omega(\Sigma(n))]\stackrel{\cong}{\longrightarrow}[\Omega(\S^{r+1}),\Omega(\Sigma(n)/G)]$$
for $r\ge 1$.\label{form}
In particular, the quotient map $\gamma_{n,\mathbb{R}} : \S^n\to \mathbb{R}P^n$ induces an isomorphism $$[\Omega(\S^{r+1}),\Omega(\S^n)]\stackrel{\cong}{\longrightarrow}[\Omega(\S^{r+1}),\Omega(\mathbb{R}P^n)]$$
for $r\ge 1$. Consequently, $$\exp_p[\Omega(\S^{r+1}),\Omega(\Sigma(n)/G)]=\exp_p[\Omega(\S^{r+1}),\Omega(\Sigma(n))]$$ and $$\exp_p[\Omega(\S^{r+1}),\Omega(\mathbb{R}P^n)]=\exp_p[\Omega(\S^{r+1}),\Omega(\S^n)]$$ for $n,r\ge 1$ and for
any prime $p$.
\end{theorem}

To state the next result, we recall that given a topological group $G$, by the Milnor's construction, there is a sequence $G\to E_1\to\cdots\to E_n\to\cdots$, where $E_n=G\ast\cdots\ast G$, the join of $(n+1)$-copies of $G$ and $EG=\mbox{colim}_n\,E_n$. Then, we have the universal $G$-fibre bundle $$G\hookrightarrow EG\stackrel{\pi}{\longrightarrow} BG=EG/G.$$
\par Next, consider the pointed suspension $\Sigma G=G\times[-1,1]/\sim$ and write $C_+=G\times[0,1]/\sim\,\subseteq \Sigma G$, and $C_-=G\times[-1,0]/\sim\,\subseteq \Sigma G$ for the upper and lower cones, respectively. Let $E=C_+\times G\sqcup C_-\times G/\sim$, where $((g,0),g')\sim((g,0),gg')$ for $g,g'\in G$.
Then, we get the principal $G$-bundle $G\to E\to\Sigma G$ and let $f: \Sigma G\to BG$ be the corresponding classifying map, and
$h : G\to \Omega(BG)$ its adjoint map. Notice that the map $h : G\to \Omega(BG)$ coincides with the composition $G\stackrel{\eta}{\to}\Omega\Sigma G\stackrel{\Omega f}{\to}\Omega (BG)$, where $\eta : G\to \Omega\Sigma G$ is determined by the unit map, i.e., the adjoint of the identity map $\mbox{id}_{\Sigma G}:\Sigma G \to \Sigma G$.

A map $\varphi : (X,\mu)\to (Y,\nu)$ of $H$-spaces is called an {\em $H$-map} provided  the diagram
$$\xymatrix{X\times X\ar[dd]_\mu\ar[rr]^{\varphi\times \varphi}&&Y\times Y\ar[dd]^{\nu}\\
\\
X\ar[rr]_\varphi&&Y}$$
is homotopy commutative.

\par If the group $G$ is a $CW$-complex (e.g., if $G$ is a Lie or discrete group) then the homotopy equivalence $X\ast Y\simeq \Sigma(X\wedge Y)$ imposes a $CW$-structure on $E_n$ and hence, on $EG$ as well. The fact that $G$ is a $CW$-complex implies that $G$ acts cellularly on $EG$. Consequently, $BG=EG/G$ is also a $CW$-complex. By the well-known Milnor's result \cite{milnor}, the space $\Omega(BG)$ is a $CW$-complex as well.

\par Let $\partial : \Omega(BG)\to G$ be the connecting map in the Barratt-Puppe sequence
$$\cdots\to \Omega(G)\to \Omega(EG)\to \Omega(BG)\stackrel{\partial}{\to} G \to EG\to BG$$
associated to the universal $G$-bundle $G\hookrightarrow EG\stackrel{\pi}{\to} BG$.

\begin{proposition}\label{H-equiv}
{\em Suppose the topological group $G$ is a $CW$-complex. Then the connecting map $\partial : \Omega(BG)\to G$ is a homotopy equivalence and is an $H$-map.}
\end{proposition}
\begin{proof} By \cite[Theorem 8.6]{stasheff}, the map $h : G\to \Omega(BG)$  is an $H$-map and is a weak homotopy equivalence. Since $G$ is
a $CW$-complex, it follows that $h$ is a homotopy equivalence. Following \cite[p.\ 409]{hatcher}, the connecting map $\partial : \Omega(BG)\to G$ is given by
$\partial=\rho^{-1}\circ j\circ \varrho$ as in the following commutative diagram
$$\xymatrix{\Omega(BG)\ar[dd]_{\varrho}\ar[rr]^{\partial}&&G\ar[dd]^{\rho}\\
\\
F_i\ar[rr]_{j}&&F_\pi,}$$
where $F_\pi$ is the homotopy fiber of the map $\pi: EG \to BG$ and $F_i$ is the homotopy fiber of the inclusion map $i: F_\pi\hookrightarrow EG$,
both maps $\Phi$ and $\rho$ are the obvious homotopy equivalences. It is easy to verify that the composition $\partial \circ h=\mbox{id}_G$,
the identity map on $G$. Since the map $h$ is an $H$-map and a homotopy equivalence, so is $\partial$ and the proof is complete.
\end{proof}

\begin{corollary}\label{adold1}
{\em Let $G$ be a compact Lie group and $G\to X\to X/G$ be a principal $G$-fibration where $X/G$ is paracompact. Then the connecting map $\partial_X:  \Omega(X/G)\to G$
in the associated Barratt-Puppe sequence $$\cdots\to \Omega(G)\to \Omega(X)\to \Omega(X/G)\stackrel{\partial_X}{\to} G \to X\to X/G$$
is an $H$-map.}
\end{corollary}
\begin{proof}
Since the principal $G$-fibration $X\to X/G$ is classified by the classifying map $\varphi: X/G\to BG$ which can be lifted to a $G$-map $\tilde \varphi:X\to EG$, we have the following commutative diagram
\begin{equation*}
\begin{CD}
{...} @>>> \Omega(X)   @>>>  \Omega(X/G)     @>{\partial_X}>>   G    @>>> X  @>>> X/G \\
@.                     @V{\Omega\tilde\varphi}VV                 @V{\Omega \varphi}VV                   @|                       @V{\tilde \varphi}VV                     @V{\varphi}VV \\
{...} @>>> \Omega(EG)   @>>>  \Omega(BG)     @>{\partial}>>   G    @>>> EG  @>>> BG.
\end{CD}
\end{equation*}
It follows that $\partial_X=\partial\circ \Omega \varphi$, where $\partial$ is as in Proposition \ref{H-equiv}.
Since $\Omega \varphi$ and $\partial$ are $H$-maps, by Proposition \ref{H-equiv}, it follows that $\partial_X$ is an $H$-map as well.
\end{proof}

Given a group $G$ and a prime $p$, denote by $\exp_p(G)$ the least positive integer $p^t$ such that $\alpha^{p^t}=1$ for any $p$-torsion elements $\alpha$ in $G$. If such an integer does not exist or if $G$ has no $p$-torsion elements then we set $\exp_p(G)=1$.
The following result is straightforward.
\begin{proposition}\label{exact-exp}
{\em Let $1\to G' \to G \to G'' \to 1$ be a short exact sequence of groups. Then for any prime $p$,
$$
\exp_p(G')\le\exp_p(G)\le \exp_p(G')\cdot \exp_p(G'').
$$}
\end{proposition}

We now give the $p$-primary exponent of $[\Omega(\S^{r+1}),\Omega(\mathbb{C}P^n)]$.

\begin{theorem} The principal $\S^1$-bundle $\S^1\to \S^{2n+1}\stackrel{\gamma_{n,\mathbb{C}}}{\longrightarrow}\mathbb{C}P^n$ gives rise to a split short exact sequence
$$1\to [\Omega(\S^2),\Omega(\S^{2n+1})]\stackrel{(\gamma_{n,\mathbb{C}})_\ast}\longrightarrow [\Omega(\S^2),\Omega(\mathbb{C}P^n)]\to \mathbb{Z}\to 0$$
of groups and an isomorphism $$[\Omega(\S^{r+1}),\Omega(\mathbb{C}P^n)]\stackrel{\cong}{\longrightarrow}[\Omega(\S^{r+1}),\Omega(\S^{2n+1})]$$
for $r\ge 2$.

Consequently, $$\exp_p[\Omega(\S^{r+1}),\Omega(\mathbb{C}P^n)]=\exp_p[\Omega(\S^{r+1}),\Omega(\S^{2n+1})]$$ for $n,r\ge 1$ and any prime $p$.
\label{complex}
\end{theorem}
\begin{proof} Consider the principal $\S^1$-bundle $\S^1\hookrightarrow \S^{2n+1}\stackrel{\gamma_{n,\mathbb{C}}}{\longrightarrow}\mathbb{C}P^n$ and
the associated Barratt-Puppe sequence
$$\cdots\to\S^1\to \Omega(\S^{2n+1})\to\Omega(\mathbb{C}P^n)\stackrel{\partial}{\to}\S^1\to\S^{2n+1}\to \mathbb{C}P^n.$$
Notice that $[\Omega(\S^2),\S^1]=H^1(\Omega(\S^2),\mathbb{Z})\cong\mathbb{Z}$ and $[\Omega(\S^{r+1}),\S^1]=0$ for $r\ge 2$.
Since the inclusion map $\S^1\hookrightarrow\S^{2n+1}$ is null-homotopic and,  by Corollary \ref{adold1}, the connecting map $\partial : \Omega(\mathbb{C}P^n)\to\S^1$ is an $H$-map we get a split short exact sequence of groups
$$1\to[\Omega(\S^2),\Omega(\S^{2n+1})]\stackrel{(\gamma_{n,\mathbb{C}})_\ast}{\longrightarrow}[\Omega(\S^2),\mathbb{C}P^n]\longrightarrow[\Omega(\S^2),\S^1]\cong\mathbb{Z}\to 0.$$
Furthermore, since $[\Omega (\S^{r+1}), \S^1]=0$ for $r\ge 2$, the Barratt-Puppe sequence splits and we obtain isomorphisms
$$[\Omega(\S^{r+1}),\Omega(\mathbb{C}P^n)]\stackrel{\cong}{\longrightarrow}[\Omega(\S^{r+1}),\Omega(\S^{2n+1})] \qquad \text{for $r\ge 2$}.$$

Now, Proposition \ref{exact-exp} yields $$\exp_p[\Omega(\S^{r+1}),\Omega(\mathbb{C}P^n)]=\exp_p[\Omega(\S^{r+1}),\Omega(\S^{2n+1})]$$ for $n,r\ge 1$ and any prime $p$,
and the proof is complete.
\end{proof}

Finally, we consider the quaternionic projective spaces $\mathbb{H}P^n$ with $n\ge 1$ and $\mathbb{H}P^\infty=\mbox{colim}_n\mathbb{H}P^n$.
Because the canonical inclusion map $i :\mathbb{H}P^1=\S^4\hookrightarrow \mathbb{H}P^\infty$ is the classifying map of the
Hopf fibration $\nu_4:\S^7\to \S^4$, there is a map of fibrations
$$\xymatrix{\S^3\ar[d]\ar@{=}[rr]&&\S^3\ar[d]\\
\S^7\ar[d]_{\nu_4}\ar[rr]&&E\S^3\ar[d]\\
\S^4\ar@{^{(}->}[rr]^i&&\mathbb{H}P^\infty.}$$
This implies that the induced map $i_\ast : \pi_k(\S^4)\to\pi_k(\mathbb{H}P^\infty)$ is surjective for $k\ge 1$. Since $\pi_k(\S^4)\cong\pi_k(\S^7)\oplus \Sigma\pi_{k-1}(\S^3)$,
we derive that the restriction ${i_\ast}_| : \Sigma\pi_{k-1}(\S^3)\to \pi_k(\mathbb{H}P^\infty)$ is an isomorphism and there is a splitting
short exact sequence $$0\to\pi_k(\S^7)\longrightarrow\pi_k(\S^4)\longrightarrow\pi_k(\mathbb{H}P^\infty)\to 0$$ for $k\ge 1$.
Now, given  $\alpha\in \pi_k(\mathbb{H}P^\infty)$, there is  $\alpha'\in \pi_{k-1}(\S^3)$ such that $\alpha=i_\ast \Sigma \alpha'$.
\begin{proposition} {\em \label{whp}
\mbox{\em (1)} If  $\alpha=i_\ast\Sigma \alpha'\in \pi_k(\mathbb{H}P^\infty)$ and  $\beta=i_\ast\Sigma \beta'\in \pi_l(\mathbb{H}P^\infty)$ then
$[\alpha,\beta]=i_\ast(\Sigma\nu'^+\circ \Sigma(\alpha'\wedge \beta'))$ and $12[\alpha,\beta]=0$.

\mbox{\em (2)} If $\alpha=i_\ast \Sigma \alpha'\in \pi_k(\mathbb{H}P^\infty)$, $\beta=i_\ast\Sigma \beta'\in \pi_l(\mathbb{H}P^\infty)$ and $\gamma=i_\ast\Sigma \gamma'\in \pi_m(\mathbb{H}P^\infty)$
then  $[[\alpha,\beta],\gamma]=\Sigma(\nu'^+)\Sigma^4(\nu'^+)\circ \Sigma(\alpha'\wedge \beta'\wedge \gamma')$ and $3[[\alpha,\beta],\gamma]=0$.

\mbox{\em (3)} If  $\alpha_i=i_\ast\Sigma \alpha'_i\in\pi_{k_i}(\mathbb{H}P^\infty)$ with $i=1,\ldots,m$ and $m\ge 4$ then all Whitehead products $[\alpha_1,\ldots,\alpha_m]=0$.}
\end{proposition}
\begin{proof}
(1): Given $\alpha\in \pi_k(\mathbb{H}P^\infty)$ and  $\beta\in \pi_l(\mathbb{H}P^\infty)$ there are $\alpha'\in \pi_{k-1}(\S^3)$ and  $\beta'\in \pi_{l-1}(\S^3)$
such that $\alpha=i_\ast \Sigma \alpha'$ and  $\beta=i_\ast \Sigma \beta'$. Hence, by means of Proposition \ref{wh}(6), we get $[\alpha,\beta]=[i_\ast \Sigma \alpha',i_\ast \Sigma \beta']=i_\ast([\iota_4,\iota_4]\circ \Sigma(\alpha'\wedge \beta'))$.
Since, in view of \cite[(1.20)]{GM1}, we have $[\iota_4,\iota_4]=2\nu_4-\Sigma\nu'^+$ for $\nu'^+=\nu'-\alpha_1(3)$, we conclude that
$$[\alpha,\beta]=i_\ast(\Sigma\nu'^+\circ \Sigma(\alpha'\wedge \beta')).$$
Since the order $|\Sigma\nu'^+|=12$, this implies that $12[\alpha,\beta]=0$.

(2): An element $\gamma\in\pi_m(\mathbb{H}P^\infty)$ leads to $\gamma'\in \pi_{m-1}(\S^3)$ with  $\gamma=i_\ast \Sigma \gamma'$.
Then, $[[\alpha,\beta],\gamma]=i_\ast[(E\nu'^+\circ \Sigma(\alpha'\wedge \beta'),\Sigma \gamma']=i_\ast[\Sigma\nu'^+,\iota_4]\circ \Sigma(\alpha'\wedge \beta'\wedge \gamma')=
\Sigma(\nu'^+)\Sigma^4(\nu'^+)\circ \Sigma(\alpha'\wedge \beta'\wedge \gamma')$.
\par But, $\nu'^+=\nu'-\alpha_1(3)$, and in view of \cite[(1.25) and (1.28)]{GM1}, we have $\Sigma^2\nu'=2\nu_5$, $\nu'\circ\nu_6=0$.
Then, we conclude that $$[[\alpha,\beta],\gamma]=i_\ast(\alpha_1(4)\alpha_1(7)\circ \Sigma(\alpha'\wedge \beta'\wedge \gamma')).$$
By \cite[(1.8)]{GM1},  the order $|\alpha_1(4)\alpha_1(7)|=3$, this implies that $3[[\alpha,\beta],\gamma]=0$.

(3): Finally, $\delta\in\pi_t(\mathbb{H}P^\infty)$ leads to $\delta'\in \pi_{t-1}(\S^3)$ with $\delta=i_\ast \Sigma \delta'$.
Then, $[[[\alpha,\beta],\gamma],\delta]=i_\ast([\alpha_1(4)\alpha_1(7)\circ \Sigma(\alpha'\wedge \beta'\wedge \gamma'),\Sigma \delta'])=
i_\ast([\alpha_1(4)\alpha_1(7),\iota_4]\circ \Sigma(\alpha'\wedge \beta'\wedge \gamma'\wedge \delta'))=
i_\ast([\iota_4,\iota_4]\Sigma^4(\alpha_1(3)\alpha_1(6))\circ \Sigma(\alpha'\wedge \beta'\wedge \gamma'\wedge \delta'))=i_\ast([\iota_4,\iota_4]\alpha_1(7)\alpha_1(10)\circ \Sigma(\alpha'\wedge \beta'\wedge \gamma'\wedge \delta'))$.
Since, in view of \cite[(1.8)]{GM1}, we have $\alpha_1(7)\alpha_1(10)=0$, we deduce that $$[[[\alpha,\beta],\gamma],\delta]=0.$$
From this we conclude that all Whitehead products $[\alpha_1,\ldots,\alpha_m]=0$ of weight $m\ge 4$ for $\alpha_i\in\pi_{k_i}(\mathbb{H}P^\infty)$ with $i=1,\ldots,m$
and the proof follows.
\end{proof}

Now, for $Y=\mathbb{H}P^\infty$ we have
\begin{theorem} If $r\ge 1$ then $$\exp_2[\Omega(\S^{r+1}),\Omega(\mathbb{H}P^\infty)]=\exp_2[\Omega(\S^{r+1}),\S^3]\le 2\exp_2(\S^4)$$ and
$$\exp_p[\Omega(\S^{r+1}),\Omega(\mathbb{H}P^\infty)]=\exp_p[\Omega(\S^{r+1}),\S^3]=\exp_p(\S^4)=p^3$$
for any odd prime $p$.\label{quater}
\end{theorem}
\begin{proof}
Since $\mathbb{H}P^\infty=B\S^3$, in view of Proposition \ref{H-equiv}, the connection map $\partial : \Omega(\mathbb{H}P^\infty)\to \S^3$ associated
with the fibration $\S^3\rightharpoonup E\S^3\to\mathbb{H}P^\infty$ leads to an isomorphism $$[\Omega(\S^{r+1}),\Omega(\mathbb{H}P^\infty)]\cong[\Omega(\S^{r+1}),\S^3]$$ for $r\ge 1$ and consequently,
$$\exp_p[\Omega(\S^{r+1}),\Omega(\mathbb{H}P^\infty)]=\exp_p[\Omega(\S^{r+1}),\S^3].$$
In the rest of proof we mimic {\em mutatis mutandis} the ideas of the proof
of Theorem \ref{main} and sketch below the main facts only.

Case (1): $p=2$. In view of Proposition \ref{whp}, we get that $4[\alpha,\beta]=0$ for $\alpha\in\pi_k(\mathbb{H}P^\infty;2)$ and $\beta\in\pi_l(\mathbb{H}P^\infty;2)$.
Further, all Whitehead products $[\alpha_1,\ldots,\alpha_m]=0$ of weight $m\ge 3$ for $\alpha_i\in\pi_{k_i}(\mathbb{H}P^\infty;2)$ with $i=1,\ldots,m$.

Because, by means of \cite{selick} and the $EHP$ sequence, we have $\exp_2(\S^4)=2^t$ with $t>1$, we conclude that $$\exp_2[\Omega(\S^{r+1}),\Omega(\mathbb{H}P^\infty)]\le 2\exp_2(\S^4)=2^{t+1}.$$

Now let $p$ be an odd prime. Then, by (\ref{expon}) we have, $\exp_p(\S^4)=p^3$.

Case (2): $p=3$. In view of Proposition \ref{whp}, we get that $3[\alpha,\beta]=0$ for $\alpha\in\pi_k(\mathbb{H}P^\infty;3)$ and $\beta\in\pi_l(\mathbb{H}P^\infty;3)$.
Further, $3[[\alpha,\beta],\gamma]=0$ for $\alpha\in\pi_k(\mathbb{H}P^\infty;3)$, $\beta\in\pi_l(\mathbb{H}P^\infty;3)$ and $\gamma\in\pi_l(\mathbb{H}P^\infty;3)$.
Certainly, all Whitehead products $[\alpha_1,\ldots,\alpha_m]=0$ of weight $m\ge 4$ for $\alpha_i\in\pi_{k_i}(\mathbb{H}P^\infty;3)$ with $i=1,\ldots,m$.
Because  $\exp_3(\S^4)=27$, we conclude that $$\exp_3[\Omega(\S^{r+1}),\Omega(\mathbb{H}P^\infty)]=\exp_3(\S^4)=27.$$

Case (2): $p>3$. In view of Proposition \ref{whp}, we get that  all Whitehead products $[\alpha_1,\ldots,\alpha_m]=0$ of weight $m\ge 2$
for $\alpha_i\in\pi_{k_i}(\mathbb{H}P^\infty;p)$ with $i=1,\ldots,m$.
Hence, we conclude that $$\exp_p[\Omega(\S^{r+1}),\Omega(\mathbb{H}P^\infty)]=\exp_p(\S^4)=p^3$$
and the proof is complete.
\end{proof}

Let $\gamma_{n,\mathbb{H}} : \S^{4n+3}\to \mathbb{H}P^n$ be the quotient map. Since $\S^3\hookrightarrow \S^{4n+3}\stackrel{\gamma_{n,\mathbb{H}}}{\to}\mathbb{H}P^n$
is a principal $\S^3$-bundle, Corollary \ref{adold1} and Proposition \ref{exact-exp} lead to:
\begin{theorem} \label{HP(n)}
The principal $\S^3$-bundle $\S^3\hookrightarrow \S^{4n+3}\stackrel{\gamma_{n,\mathbb{H}}}{\to}\mathbb{H}P^n$ gives rise to a short exact sequence
$$1\to [\Omega(\S^{r+1}),\Omega(\S^{4n+3})]\stackrel{(\gamma_{n,\mathbb{H}})_\ast}\longrightarrow [\Omega(\S^{r+1}),\Omega(\mathbb{H}P^n)]\to [\Omega(\S^{r+1}),\S^3]\to 1$$
of groups for $r\ge 1$.
Consequently, $$\exp_p[\Omega(\S^{r+1}),\Omega(\S^{2n+1})]\le\exp_p[\Omega(\S^{r+1}),\Omega(\mathbb{H}P^n)]\le \exp_p[\Omega(\S^{r+1}),\Omega(\S^{2n+1})]\cdot \exp_p[\Omega(\S^{r+1}),\S^3]$$ for $n,r\ge 1$ and any prime $p$.
\label{quaternionic}
\end{theorem}

\begin{remark}
In Theorem \ref{HP(n)}, the bounds can be computed or estimated using Theorem \ref{main} together with fact that $\exp_p[\Omega(\S^{r+1}),\S^3]=\exp_p(\S^4)=p^3$ if $p>2$.
While we give bounds for $\exp_p[\Omega(\S^{r+1}),\Omega(\mathbb{H}P^n)]$, the calculation of these exponents will be carried out for $n,r\ge 1$ and any prime $p$, in a forthcoming paper.
\end{remark}

\centerline{\sc Acknowledgments} This work was initiated and completed during the authors' visits to Banach Center in Warsaw, Poland,
October 27 - November 05, 2016 and February 17 - March 03, 2018, respectively. The authors would like to thank
the Banach Center in Warsaw, Poland and the Faculty of Mathematics and Computer Science, the University of Warmia and Mazury in Olsztyn, Poland
for their hospitality and support. 
\par Special thanks are due to Jim Stasheff for pointing out the Dold-Lashof result in \cite{dold} and \cite{stasheff} and to Jie Wu for helpful conversations regarding the Cohen groups.
Finally, the authors would like to express their gratitude to three anonymous referees for their invaluable suggestions which help improve the exposition of the paper.

\end{document}